\documentclass{amsart}
\usepackage{latexsym}
\usepackage{amssymb}
\usepackage{amsfonts}
\usepackage{amsmath}
\usepackage{amsthm}
\usepackage{graphics}
\usepackage[pdftex]{graphicx}

\newtheorem{lemma}{Lemma}
\newtheorem{theorem}[lemma]{Theorem}

\newtheorem{corollary}[lemma]{Corollary}

\newtheorem{remark}[lemma]{Remark}

\newtheorem{conj}{Conjecture}
\def\cs#1{{\rm cs}(#1)}
\def\De{\Delta}
\def\Ga{\Gamma}
\def\diam{{\rm diam}}

\def\gcd{{\rm gcd}}
\begin{document}
\openup 2\jot

\title[Divisibility graph for symmetric and alternating groups]
{Divisibility graph for symmetric and alternating groups}
\author[A.~Abdolghafourian]{Adeleh~Abdolghafourian}
\address{A.~Abdolghafourian, Department of Mathematics \newline Yazd University
\\ Yazd, 89195-741, Iran}
\email{A.Abdolghafourian@stu.yazd.ac.ir}
\author[M.~A.~Iranmanesh]{Mohammad~A.~Iranmanesh}
\thanks{Corresponding author: M.~A.~Iranmanesh}\address{M.~A.~Iranmanesh, Department of Mathematics \newline Yazd University\\ Yazd, 89195-741 , Iran}
\email{iranmanesh@yazd.ac.ir}

\begin{abstract}
Let $X$ be a non-empty set of positive integers and $X^*=X\setminus \{1\}$. The divisibility
graph $D(X)$ has $X^*$ as the vertex set and there is an edge connecting $a$ and $b$ with $a, b\in X^*$
whenever $a$ divides $b$ or $b$ divides $a$. Let $X=\cs{G}$ be the set of conjugacy class sizes
of a group $G$. In this case, we denote $D(\cs{G})$ by $D(G)$. In this paper we will find the number
of connected components of $D(G)$ where $G$ is the symmetric group $S_n$ or is the alternating
group $A_n$.
\end{abstract}
\openup 1.2\jot
\maketitle
\section{Introduction}
\label{sec:introd}
There are several graphs associated to various algebraic structures, especially finite groups, and many interesting
results have been obtained recently, as for example, in~\cite{A, CG, CD, VS}.

Let $X$ be a set of positive integers and $X^*=X\setminus \{1\}$. Mark~L.~Lewis in~\cite{L} introduced
two graphs associate with $X$, the common divisor graph and the prime vertex graph. The ~\emph{common divisor graph} $\Ga(X)$ is a graph with vertex set $V(\Ga(X)) = X^{*}$, and edge set $E(\Ga(X)) = \{\{x,y\} : \gcd(x,y)\neq1\}$. Note that $\gcd(x, y)$
denotes the greatest common divisor of $x, y$. The ~\emph{prime vertex graph}
$\De(X)$ is a graph with vertex set $V(\De(X)) = \rho(X) = \bigcup_{x\in X}\pi(x)$, where
$\pi(x)$ is the set of primes dividing $x$ and edge set
$E(\De(G)) = \{\{p,q\} : pq~\text{divides}~x,\ x\in X\}$.

Praeger and the second author defined a bipartite graph $B(X)$ in~\cite{IP} and elucidated the connection
between these graphs. The ~\emph{bipartite divisor graph} $B(X)$ is a graph with the vertex set
$V(B(X)) = \rho(X) \bigcup X^{*}$, and the edge set
$E(B(X)) = \{\{p,x\} : p\in\rho(X), x\in X^{*}~\text{and}~ p~\text{divides}\ x\}$.

Recently A.~R.~Camina and R.~D.~Camina in~\cite{CC} introduced a new directed graph (or simply a digraph)
using the notion of divisibility of positive numbers. The~\emph{divisibility digraph} $\overrightarrow{D}(X)$
has $X^*$ as the vertex set and there is an arc connecting $(a, b)$ with $a, b\in X^*$
whenever $a$ divides $b$. It is clear that the digraph $\overrightarrow{D}(X)$ is not strongly connected (where
by strongly connected digraph we mean a digraph such that there exists a directed path between any of two vertices).
So it is important to find the number of connected components of its underlying graph. We denote the
underlying graph of $\overrightarrow{D}(X)$ by $D(X)$. By the~\textit{diameter} of a graph $\Omega$, $\diam(\Omega)$, we mean
the maximum diameter of its connected components.

For a finite group $G$ and $g\in G$, let $g^G= \{x^{-1}gx : x\in G\}$ be the conjugacy class
of $g$ in $G$ and $\cs{G}=\{|g^G| : g\in G\}$ be the set of conjugacy class sizes of $G$. If $\delta$ is a
permutation, then the cycle decomposition of $\delta$ is its expression as a product of disjoint cycles.
It is interesting to investigate the properties of the prime vertex graph, the common divisor graph, the bipartite divisor
graph and the divisibility graph when $X=\cs{G}$. In this case we denote $\Ga(\cs{G})$, $\De(\cs{G})$,
$B(\cs{G})$ and $D(\cs{G})$ by $\Ga(G), \De(G), B(G)$ and $D(G)$ respectively. For properties of $\Ga(G), \De(G)$ and $B(G)$
we refer to~\cite{BHM, BDIP, CD, K}.

In this paper we investigate the graph $D(G)$ where $G$ is the symmetric group $S_n$ or the alternating
group $A_n$. Note that two vertices $a, b$ of this graph are adjacent if either $a$ divides $b$ or $b$
divides $a$. Let $\delta\in S_n$. Suppose that there are $k_i$ cycles of length $m_i$ ($1\leq i\leq r$)
in the cycle decomposition of $\delta$, such that $m_i\neq 1$ and $m_i\neq m_j$, for $1\leq i,j\leq r$, then we denote it by $\delta=[1^t,m_1^{k_{1}},...,m_r^{k_{r}}]$ where $t={n-\sum_{i=1}^{r}k_im_i}$. Also we denote the vertex corresponding to $|g^G|$ in $D(G)$ by $v_g$. Throughout the paper, $p$ is a prime number.

In Section~\ref{sec:preli} we recall some basic lemmas and theorems which we need in the next sections.
In Section~\ref{sec:div1} we will find the number of connected components of $D(S_n)$. The main theorem
of this section is Theorem~\ref{thm:9}. In Section~\ref{sec:div2} we will find the number of connected
components of $D(A_n)$. The main theorem of this section is Theorem~\ref{thm:13}.

\section{preliminaries}
\label{sec:preli}
In this section we recall some basic technical facts that we will use later.
See ~\cite[Chapter~13]{C}, ~\cite[Chapter 4]{CL} or ~\cite[p.131]{DF} for proofs and details.
\begin{lemma}\label{lem:1}
Suppose that $\delta=[1^t,m_1^{k_{1}},...,m_r^{k_{r}}]$ where $t={n-\sum_{i=1}^{r}k_im_i}$,
then $|C_{S_n}(\delta)|=(\prod_{i=1}^rk_i!{m_i}^{k_i})t!$.
\end{lemma}
\begin{lemma}\label{lem:3}
Let $\delta\in A_n$, then there is an odd permutation in $C_{S_n}(\delta)$ if and only if $|C_{S_n}(\delta)|=2|C_{A_n}(\delta)|$.
\end{lemma}
\begin{corollary}\label{cor:4}
Let $\delta\in A_n$. Then either $|\delta^{S_n}|=|\delta^{A_n}|$ or $|\delta^{S_n}|=2|\delta^{A_n}|$.
\end{corollary}
\begin{lemma}\label{lem:5}
Let $\delta\in A_n$ fixes at least two points. Then $|C_{S_n}(\delta)|=2|C_{A_n}(\delta)|$ and $|\delta^{A_n}|=|\delta^{S_n}|$.
\end{lemma}

\begin{lemma}\label{lem:6}
Suppose that $\delta=[1^t,m_1^{k_{1}},...,m_r^{k_{r}}]$ and $t\leq1$. If there exists $i$
such that $m_i$ is even or there exists $i$ such that $m_i$ is odd
and $k_i\geq2$, then $|C_{S_n}(\delta)|=2|C_{A_n}(\delta)|$ and
$|\delta^{A_n}|=|\delta^{S_n}|$.
\end{lemma}

By Lemma~\ref{lem:5} and Lemma~\ref{lem:6} we have the following corollary.
\begin{corollary}\label{cor:7}
If $\delta=[1^t,m_1^1,...,m_r^1]\in A_n$, each $m_i$ is
odd and $t\leq1$, then we have $|C_{A_n}(\delta)|=|C_{S_n}(\delta)|$ and
$|\delta^{A_n}|=\frac{1}{2}|\delta^{S_n}|$. For other cases, we have $|C_{A_n}(\delta)|=\frac{1}{2}|C_{S_n}(\delta)|$
and $|\delta^{A_n}|=|\delta^{S_n}|$.
\end{corollary}
\begin{lemma}\label{lem:2}
Let $x=\sum_{i=1}^rk_im_i$, then $\prod_{i=1}^rk_i!{m_i}^{k_i}$ divides $x!$. In addition if  $m_i\geq 3$, for some $1\leq i\leq r$,
then $2\prod_{i=1}^rk_i!{m_i}^{k_i}$ divides $x!.$
\end{lemma}
\begin{proof}
We prove this lemma by induction on $r$. Let $r=1$, then
$$x!=\prod_{\substack{i=0}}^{\scriptstyle{k_1m_1-1}}(k_1m_1-i )= k_1!m_1^{k_1}\prod_{\substack{i=1 \\ m_1\nmid i}}^{\scriptstyle{k_1m_1-1}}(k_1m_1-i ).$$
If $m_1\geq 3$, then $2$ divides $\prod\limits_{\substack{i=1 \\ m_1\nmid i}}^{\scriptstyle{k_1m_1-1}}(k_1m_1-i ) $. Therefore $2k_1!m_1^{k_1}$ divides $x!.$

Suppose that $r=t$. So $x=\sum_{i=1}^tk_im_i=x'+k_tm_t$ where
$x'=\sum_{i=1}^{t-1}k_i m_i$.
Since ${x\choose k_t m_t}\in \mathbb{N}$, we conclude that
$x'!(k_tm_t)!$ divides $x!$. By induction hypothesis $\prod_{i=1}^{t-1}k_i!{m_i}^{k_i}$
divides $x'!$ and $k_t!m_{t}^{k_t}$ divides $(k_tm_t)!.$  Let $m_i\geq 3$ for some $1\leq i\leq r$.
Then, without loss of generality, we may assume $m_t\geq 3$. So  $2k_t!m_t^{k_t}$ divides $(k_tm_t)!.$
Therefore $2\prod_{i=1}^{t}k_i!{m_i}^{k_i}$ divides $x!.$
\end{proof}
\begin{remark}\label{rem:0}
Let $G$ be a finite group and $x\in G$. It is clear that $C_G(x)\leq C_G(x^m)$ for every natural number $m$. So $|(x^m)^G|$ divides $|x^G|$. This means that $v_{x}$ is adjacent to $v_{x^m}$ in $D(G)$.
\end{remark}

\section{Divisibility graph for $S_n$}
\label{sec:div1}
In this section we investigate the number of connected components of $D(S_n)$. We will prove
that $D(S_n)$ has at most two connected components. If it is disconnected, then one of its
connected components is an isolated vertex, that is, a copy of $K_1$.

It is easy to see that both $D(S_1)$ and $D(S_2)$ are null graphs, that is, have no vertices.
Also for $n=3,~4$ and $5$, $D(S_n)$ has two connected components (see Figure~\ref{fig:1}).
\begin{figure}[here]
\begin{center}
\includegraphics[height=4cm]{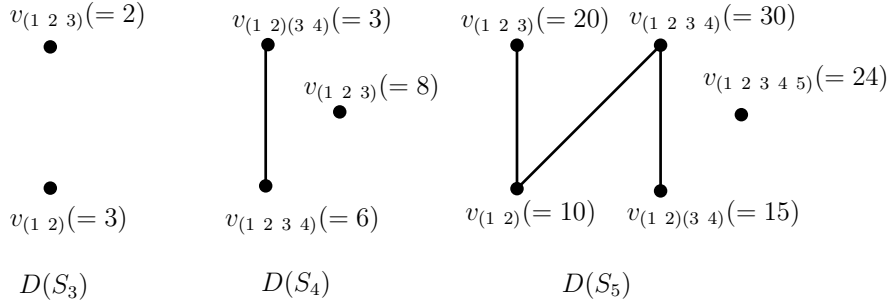}
\caption{The graph $D(S_n)$ for $n=3,~4$ and $5$.}
\label{fig:1}
\end{center}
\end{figure}

\begin{lemma}\label{lem:8}
Let $1\neq \delta\in S_n$, $n>2$ and $p\geq n-1$. Then $p$ divides $|C_{S_n}(\delta)|$ if and only if
$\delta$ is a cycle of length $p$, that is, $\delta=[1,p]$ or $\delta=[p]$.
\end{lemma}
\begin{proof}
First suppose $p$ divides $|C_{S_n}(\delta)|$. Assume, to the contrary,
that
$\delta=[1^t, m_1^{k_1},$
 $\cdots, m_r^{k_r}]$ is not a $p$-cycle. By Lemma \ref{lem:1}, $|C_{S_n}(\delta)|=(\prod_{i=1}^rk_i!{m_i}^{k_i})t!$. Since $p$ divides
$|C_{S_n}(\delta)|$, then either $p$ divides $t!$ or there exists $j$ such that
$p$ divides $m_j$ or $k_j!$. First assume that $p$ divides $m_j$. In this case we conclude that $p=m_j$.
Hence $\delta$ is a cycle of length $p$ which is a contradiction. Now suppose that $p$ divides $k_j!$.
Since $p$ is prime, we have $k_j\geq p\geq n-1$. Thus $k_jm_j\geq 2n-2>n$, which is a contradiction too.
Finally if $p$ divides $t!$, then $t=n-\sum_{i=1}^{r}k_im_i\geq p\geq n-1$, a contradiction.

For the other direction, note that if $\delta$ is a cycle of length $p$ then $|C_{S_n}(\delta)|=p(n-p)!$.
\end{proof}
\begin{theorem}\label{thm:9}
Let $1\neq\delta\in S_n$ and $n>6$. If $\delta$ is a $p$-cycle where $p\geq n-1$, then $v_{\delta}$ is an
isolated vertex of $D(S_n)$. The other vertices are in a single connected component.
\end{theorem}
\begin{proof}
First suppose that $\delta$ is a cycle of length $p=n-i$ where $i\in \{0,1\}$. Assume, to the contrary,
that $v_{\delta}$ has a neighbor, say $v_{\delta'}$, where the cycle decomposition of $\delta'$ is not
similar to $\delta$ and  $|C_{S_n}(\delta')|=x$. Then $n-i$ divides $x$, which is a contradiction by Lemma~\ref{lem:8}.
Therefore in this case $v_{\delta}$ is an isolated vertex of $D(S_n)$.

Let $v_{\tau}$ be the vertex of $D(S_n)$ corresponding to an arbitrary transposition
namely $\tau$. We prove that there exists a path between other arbitrary vertices of $D(S_n)$ and $v_{\tau}$ by using Lemma~\ref{lem:1} and Lemma~\ref{lem:2}. Since for every $\delta\in S_n$ there exists a natural number $m$ such that $\delta^m=[1^t,p^{t'}]$, by Remark~\ref{rem:0} it is enough to consider $\delta=[1^t,p^{t'}]$.
 So we have to consider three possible cases as follows:

\item[(i)] $\delta=[1^{n-2k},2^k]$ and $k\geq2$.

Let $\delta'=[1^{n-2k},4^1,2^{(k-2)}]$,
$$\frac{|C_{S_n}(\delta)|}{|C_{S_n}(\delta')|}=\frac{2^kk!(n-2k)!}{4(k-2)!2^{k-2}(n-2k)!}=k(k-1)\in \mathbb{N}.$$
Since $|C_{S_n}(\delta')|$ divides $|C_{S_n}(\delta)|$, we conclude $|\delta^{S_n}|$ divides $|\delta'^{S_n}|$. Hence $v_{\delta}$ is adjacent to $v_{\delta'}$. Also by Lemma~\ref{lem:2}, there exists a positive integer $s$ such that
$$\frac{|C_{S_n}(\tau)|}{|C_{S_n}(\delta')|}=\frac{2(n-2)!}{4(k-2)!2^{k-2}(n-2k)!}=\frac{(n-2)!}{2^{k-1}(k-2)!(n-2k)!}={n-2\choose 2(k-1)}s\in \mathbb{N}.$$
So $v_{\delta'}$ is adjacent to $v_{\tau}$ and there is a path of length two between $v_{\delta}$ and $v_{\tau}$.

\item[(ii)] $\delta=[1^{n-kp},p^k]$, $p\neq 2$ and $kp\leq n-2$.

Let $\delta'=(\alpha\ \beta)\delta$, where $\alpha$ and $\beta$ are two points fixed by $\delta$. So $\delta'=[1^{n-kp-2},2^1,p^k]$ and we obtain
$$\frac{|C_{S_n}(\delta)|}{|C_{S_n}(\delta')|}=\frac{k!p^{k}(n-kp)!}{2.k!p^{k}.(n-kp-2)!}
=\frac{(n-kp)(n-kp-1)}{2}\in \mathbb{N}.$$ Hence $v_{\delta}$ is adjacent to $v_{\delta'}$. Also by Lemma~\ref{lem:2}, there exists a positive integer $s$ such that
$$\frac{|C_{S_n}(\tau)|}{|C_{S_n}(\delta')|}=\frac{2(n-2)!}{2k!p^k(n-kp-2)!}
=\frac{(n-2)!}{k!p^k(n-kp-2)!}\\={n-2\choose kp}s\in \mathbb{N}.$$
So there is a path of length $2$ between $v_{\delta}$ and $v_{\tau}$.

\item[(iii)] $\delta=[1^{n-pk},p^k]$, $p\neq 2$ and $kp> n-2$.

In this case we have the following three subcases:

\item [1)] $k\geq 3$.

Let $\delta'=[1^{n-pk},(2p)^1, p^{(k-2)}]$,
$$\frac{|C_{S_n}(\delta)|}{|C_{S_n}(\delta')|}=\frac{k!p^k(n-kp)!}{2p(k-2)!p^{k-2}(n-kp)!}=\frac{pk(k-1)}{2}\in \mathbb{N}.$$
Again we can conclude that $|\delta^{S_n}|$ divides $|\delta'^{S_n}|$ and so $v_{\delta}$ is adjacent to $v_{\delta'}$. Since $p\neq 2$, we have $(k-1)p<kp-2$. Therefore $((k-1)p)!$ divides $(kp-2)!$. Hence by this fact and  Lemma~\ref{lem:2}, we can find positive integers $s$ and $s'$ such that
\setlength\arraycolsep{1.4pt}
\begin{eqnarray*}
\frac{|C_{S_n}(\tau)|}{|C_{S_n}(\delta')|}&=&\frac{2(n-2)!}{2p(k-2)!p^{k-2}(n-kp)!}=\frac{(n-2)!}{(k-2)!p^{k-1}(n-kp)!}\\&=&
\frac{(k-1)(n-2)!}{(k-1)!p^{k-1}(n-kp)!}=\frac{s(k-1)(n-2)!}{((k-1)p)!(n-kp)!}=\frac{s s'(k-1)(n-2)!}{(kp-2)!(n-kp)!}\\&=&‎s s'(k-1){n-2\choose kp-2}\in \mathbb{N}.
\end{eqnarray*}
This means that $v_{\delta'}$ is adjacent to $v_{\tau}$ and there exists a path of length two between $v_{\delta}$ and $v_{\tau}$.
\item [2)] $k=2$.

Let $\delta'=[1^{n-p},p^1]$. Since $p>2$, there exists a positive integer $s$ such that
$$\frac{|C_{S_n}(\delta')|}{|C_{S_n}(\delta)|}=\frac{p(n-p)!}{2p^2(n-2p)!}=\frac{(n-p)!}{2p(n-2p)!}={n-p\choose p}s\in \mathbb{N}.$$
Thus we can conclude $|\delta'^{S_n}|$ divides $|\delta^{S_n}|$. Hence $v_{\delta}$ is adjacent to $v_{\delta'}$. Also $p\leq n-2$, so by case(ii) there is a path of length two between $v_{\delta'}$ and $v_{\tau}$.
\item [3)] $k=1$.

In this case $\delta$ is a $p$-cycle. So by Lemma~\ref{lem:8}, $v_{\delta}$ is an isolated vertex.
\end{proof}
\begin{corollary}\label{cor:2}
$D(S_n)$ has at most two connected components. If it is disconnected then one of its connected components is $K_1$.
\end{corollary}
\begin{proof}
We know that for $n\geq 6$, at most one of $n$ or $n-1$ is a prime. By Theorem~\ref{thm:9} and Figure~\ref{fig:1}, we obtain the result.
\end{proof}

\section{Divisibility graph for $A_n$}
\label{sec:div2}
In this section we consider the divisibility graph for the alternating group $A_n$. We will
show that $D(A_n)$ has at most three connected components and if it is disconnected then
two of its connected components are $K_1$.
We denote $|C_{A_n}(\delta)|=(\frac{1}{2})^\sharp x$ when we
do not know whether $|C_{A_n}(\delta)|=\frac{1}{2}x $ or $
|C_{A_n}(\delta)|=x$ for some $ x\in \mathbb{N}$.
\begin{remark}\label{rem:12}
It is easy to see that $D(A_1)$, $D(A_2)$ and $D(A_3)$ are null graphs. By using \textsf{GAP}~\cite{GAP} one can see that for
$n=4,~5,~6,~7$ and $8$, $D(A_n)$ has at most three connected components (see Figure~\ref{fig:2}).
\begin{figure}[here]
\begin{center}
\includegraphics[height=16cm]{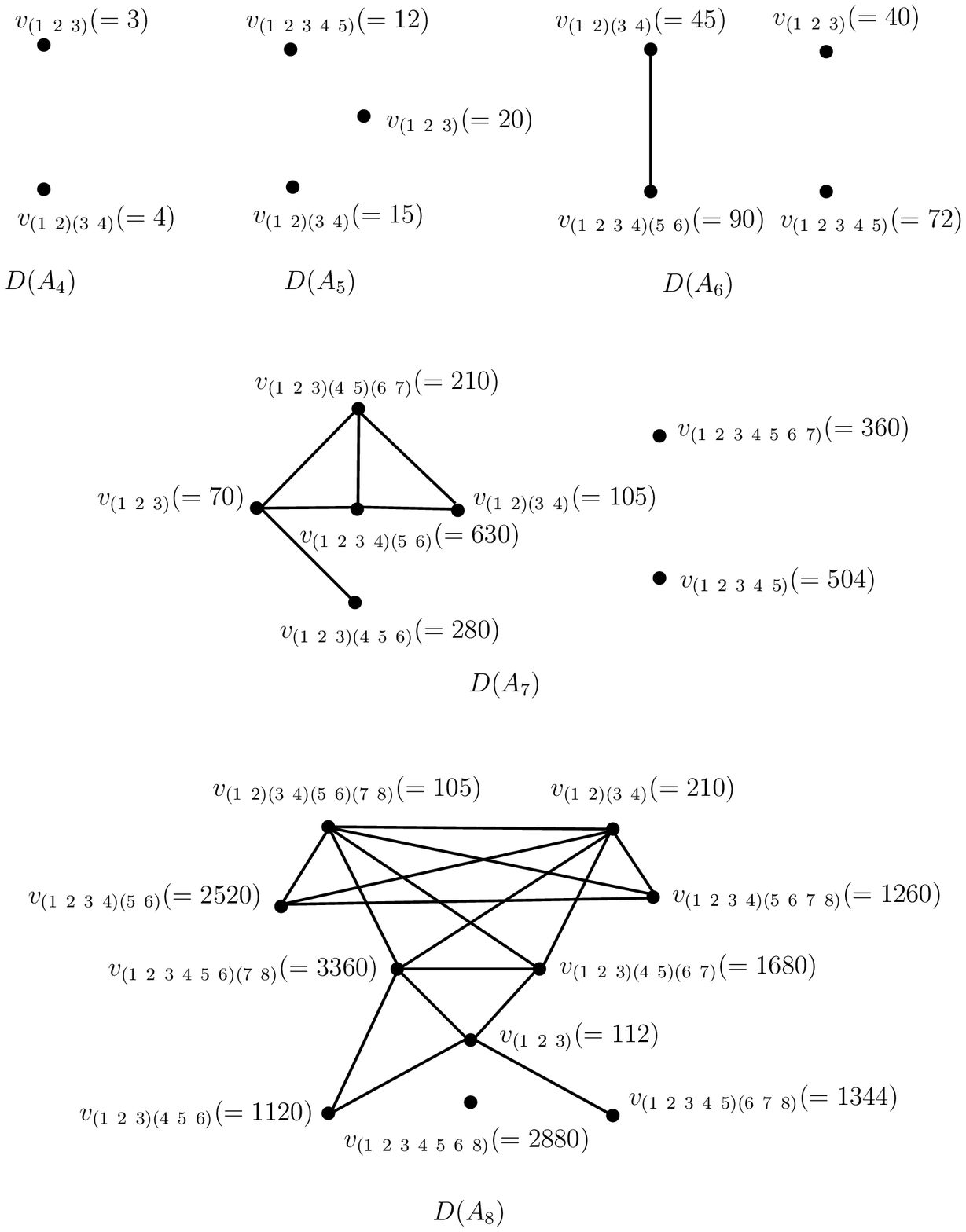}
\caption{The graph $D(A_n)$ for $n=4, 5, 6, 7$ and $8$.}
\label{fig:2}
\end{center}
\end{figure}
\end{remark}
In the rest of this section let $n\geq 9$.
\begin{lemma}\label{lem:11}
Let $1\neq\delta\in A_n$ and $p\geq n-2$. Then $p$ divides $|C_{A_n}(\delta)|$ if and only if $\delta$ is a cycle of length $p$, that is, $\delta=[1^2,p]$, $\delta=[1,p]$ or $\delta=[p]$.
\end{lemma}
\begin{proof}
First we assume that $p$ divides $|C_{A_n}(\delta)|$. Suppose, to the contrary, that $\delta=[1^t,m_1^{k_{1}},...,m_r^{k_{r}}]$ is not a $p$-cycle. By Lemma~\ref{lem:1}, Lemma~$\ref{lem:3}$ and Corollary~$\ref{cor:7}$, $|C_{A_n}(\delta)|=(\frac{1}{2})^\sharp(\prod_{i=1}^rk_i!{m_i}^{k_i})t!$. Since $p$ divides $|C_{A_n}(\delta)|$ we conclude that either $p$ divides $t!$ or there exists $j$ such that $p$ divides $m_j$ or $k_j!$.
If $p$ divides either $m_j$ or $k_j!$, then a similar argument to the proof of Lemma~\ref{lem:8} shows that
in this case either $p=m_j$ or $k_jm_j\geq 2n-4>n$, which is a contradiction.
If $p$ divides $t!=(n-\sum_{i=1}^{r}k_im_i)!$, then $\sum_{i=1}^{r}k_im_i\leq2$. Hence $\delta$
should be a transposition, which is a contradiction too.

Now let $\delta$ be a cycle of length $p$. Note that $p$ is an odd number and every cycle of length $p$ is an even permutation. In this case $|C_{A_n}(\delta)|=(\frac{1}{2})^\sharp p(n-p)!.$
Hence $p$ divides $|C_{A_n}(\delta)|$.
\end{proof}
Before proving the main theorem of this section we shall prove two lemmas.
\begin{lemma}\label{lem:14}
Let $1\neq\delta=[1^t,m_1^{k_{1}},...,m_r^{k_{r}}]\in A_n$. If there exists $j$ such that $k_j=1$ and $m_j=3$ then $v_{\delta}$ is adjacent to  $v_{(1\ 2\ 3)}$ in $D(A_n)$.
\end{lemma}
\begin{proof}
Without loss of generality we may assume that $k_1=1$ and $m_1=3$. Let $x=\sum_{i=1}^{r}k_im_i$.
By Corollary~\ref{cor:7} and Lemma~\ref{lem:2}, there exists a positive integer $s$ such that
\setlength\arraycolsep{1.4pt}
\begin{eqnarray*}
\frac{|C_{A_n}((1\ 2\ 3))|}{|C_{A_n}(\delta)|}&=&
\frac{\frac{1}{2}.3.(n-3)!}{(\frac{1}{2})^\sharp.3.(\prod_{i=2}^rk_i!{m_i}^{k_i})(n-x)!}=\frac{(n-3)!}{(2)^\sharp(\prod_{i=2}^rk_i!{m_i}^{k_i})(n-x)!}
\\&=&{n-3\choose x-3}s\in \mathbb{N}.
\end{eqnarray*}
Note that by Corollary~\ref{cor:7},  if $2$ appears in
denominator then we would have $t\leq 1,$ $~k_i = 1$ and $m_i$ odd for $1\leq i\leq r$. Since $n\geq 9$, there
exists $i$ such that $m_i\geq 5$, so by Lemma~\ref{lem:2}, we can remove $``2"$ from the denominator.

Thus $|C_{A_n}(\delta)|$ divides $|C_{A_n}((1\ 2\ 3))|$. Therefore $|(1\ 2\ 3)^{A_n}|$ divides $|\delta^{A_n}|$. So $v_{\delta}$ is adjacent to $v_{(1\ 2\ 3)}$.
\end{proof}
\begin{lemma}\label{lem:15}
Let $1\neq\delta=[1^t,m_1^{k_{1}},...,m_r^{k_{r}}]\in A_n$. If $t\geq 3$ and for each $i$, $m_i\neq3$ then there is a path of length two between $v_{\delta}$ and $v_{(1\ 2\ 3)}$ in $D(A_n)$.
\end{lemma}
\begin{proof}
Let $\delta'=(\alpha\ \beta\ \gamma)\delta$, where $\alpha$, $\beta$ and $\gamma$ are three points fixed by $\delta$.
By Corollary~\ref{cor:7},
$$\frac{|C_{A_n}(\delta)|}{|C_{A_n}(\delta')|}=\frac{\frac{1}{2}(\prod_{i=1}^rk_i!{m_i}^{k_i})t!}{(\frac{1}{2})^\sharp.3.
(\prod_{i=1}^rk_i!{m_i}^{k_i})(t-3)!}\\=\frac{t(t-1)(t-2)}{(2)^\sharp.3}\in \mathbb{N}.$$
This implies that $|C_{A_n}(\delta')|$ divides $|C_{A_n}(\delta)|$ and hence $|\delta^{A_n}|$ divides $|\delta'^{A_n}|$.
So $v_{\delta}$ is adjacent to $v_{\delta'}$ and by Lemma~\ref{lem:14}, $v_{\delta'}$ is adjacent to $v_{(1\ 2\ 3)}$.
\end{proof}
\begin{theorem}\label{thm:13}
Let $1\neq\delta\in A_n$. If $\delta$ is a $p$-cycle where $p\geq n-2$ then $v_{\delta}$ is an
isolated vertex of $D(A_n)$. The other vertices are in a single connected component.
\end{theorem}
\begin{proof}
First we show that if $\delta$ is a cycle of length $p$ where $p\geq n-2$, then $v_{\delta}$ is an isolated vertex.
Let $p=n-i$ for $i\in \{0,1,2\}$. Then $|C_{A_n}(\delta)|=n-i$. Suppose $v_{\delta}$ has a
neighbor say $v_{\delta'}$, such that the cycle decomposition of $\delta'$ is not the same as $\delta$. Let $|C_{A_n}(\delta')|=x$.
In this case it is easy to see that~$n-i$ divides $x$ which is impossible by Lemma~\ref{lem:11}.

Now we are ready to show that the other vertices of $D(A_n)$ are all in the same connected component. We show that there
exists a path between any other arbitrary vertex and the vertex corresponding to an arbitrary $3$-cycle namely $v_{\tau}$.
By Lemma~\ref{lem:1} and Corollary~\ref{cor:7}, $|C_{A_n}(\tau)|=\frac{1}{2}.3.(n-3)!$. We will use Lemma~\ref{lem:1}, Corollary~\ref{cor:7} and Lemma~\ref{lem:2} for our calculation. As for $S_n$, when $\delta\in A_n$ there is a natural number $m$ such that $\delta^m=[1^t,p^{t'}]$, so by Remark~\ref{rem:0} it is enough to consider $\delta=[1^t,p^{t'}]$.
There are the following three possible cases:

\item[(i)] $\delta=[1^{n-3k},3^k]$ and $k\geq2$.

If $k\geq3$ then let $\delta'=[1^{n-3k},9^1,3^{(k-3)}]$. Obviously  $\delta'\in A_n$ and we obtain
$$\frac{|C_{A_n}(\delta)|}{|C_{A_n}(\delta')|}=\frac{\frac{1}{2}.3^k.k!(n-3k)!}{(\frac{1}{2})^\sharp.9.(3^{k-3})(k-3)!(n-3k)!}
=\frac{3k(k-1)(k-2)}{(2)^\sharp}\in \mathbb{N}.$$
Since $|C_{A_n}(\delta')|$ divides $|C_{A_n}(\delta)|$ we conclude $|\delta^{A_n}|$ divides $|\delta'^{A_n}|$. So $v_{\delta}$ is adjacent to $v_{\delta'}$. Also by Lemma~\ref{lem:2}, there exists a positive integer $s$ such that
\setlength\arraycolsep{1.4pt}
\begin{eqnarray*}
\frac{|C_{A_n}(\tau)|}{|C_{A_n}(\delta')|}&=&\frac{\frac{1}{2}.3.(n-3)!}{(\frac{1}{2})^\sharp.9.(3^{k-3})(k-3)!(n-3k)!}
=\frac{(n-3)!}{{(2)^\sharp}.3^{k-2}(k-3)!(n-3k)!}\\&=&s{n-3\choose n-3k}\in \mathbb{N}.
\end{eqnarray*}
So $v_{\delta'}$ is adjacent to $v_{\tau}$.

If $k=2$ then,
$$\frac{|C_{A_n}(\tau)|}{|C_{A_n}(\delta)|}=\frac{3(n-3)!}{18(n-6)!}=\frac{(n-3)(n-4)(n-5)}{6}\in \mathbb{N}.$$
Again we can obtain $|\tau^{A_n}|$ divides $|\delta^{A_n}|$. So $v_{\delta}$ is adjacent to $v_{\tau}$.

\item[(ii)] $\delta=[1^{n-kp},p^k]$, $p\neq3$ and $k>1$.

Note that if $kp\leq n-3$ then $\delta$ satisfies conditions of Lemma~\ref{lem:15}. So suppose $kp>n-3$.

We consider the following five subcases:

\item[1)] $k\geq 4$ and $p\neq2$. In this case let $\delta'=[1^{n-kp},(2p)^2,p^{(k-4)}]\in A_n$ and $x=kp$.
$$\frac{|C_{A_n}(\delta)|}{|C_{A_n}(\delta')|}=\frac{\frac{1}{2}k!p^k(n-kp)!}{\frac{1}{2}.2.{(2p)^2}(k-4)!p^{k-4}(n-kp)!}=
\frac{k!p^2}{8(k-4)!}\in \mathbb{N}.$$
This yields that $|\delta^{A_n}|$ divides $|\delta'^{A_n}|$. So $v_{\delta}$ is adjacent to $v_{\delta'}$.

Let $\delta''=[1^{n-2p-3},3^1,p^2]$. By Lemma \ref{lem:2} and this fact that $p\geq 5$, we can find positive integers $s$ and $s'$  such that‎
\setlength\arraycolsep{1.4pt}
\begin{eqnarray*}
\frac{|C_{A_n}(\delta'')|}{|C_{A_n}(\delta')|}&=&\frac{3(n-2p-3)!}{4(k-4)!p^{k-4}(n-kp)!}=\dfrac{3(n-2p-3)! s}{2(kp-4p)!(n-kp)!}\\&=&\dfrac{3(n-2p-3)!s s'}{(kp-2p-3)!(n-kp)!}=3s s'{n-2p-3\choose kp-2p-3}\in \mathbb{N}.
‎\end{eqnarray*}
So $v_{\delta'}$ is adjacent to $v_{\delta''}$. Now $\delta''$ satisfies conditions of Lemma~\ref{lem:14}. Therefore there is a path of length three between $v_{\delta}$ and $v_{\tau}$.
\item[2)] $k>4$ and $p=2$. Since $\delta\in A_n$, in this case we must have $k\geq 6$. Let $\delta'=[1^{n-2k},(2k-2)^1,2^1]$.
$$\frac{|C_{A_n}(\delta)|}{|C_{A_n}(\delta')|}=\frac{\frac{1}{2}.k!2^k(n-2k)!}{\frac{1}{2}.2(2k-2)(n-2k)!}\in \mathbb{N}.$$
So $v_{\delta}$ is adjacent to $v_{\delta'}$. Also let $\delta''=[1^{n-7},2^2,3^1]$.
$$\frac{|C_{A_n}(\delta'')|}{|C_{A_n}(\delta')|}=\frac{8.3.(n-7)!}{2(2k-2)(n-2k)!}\in \mathbb{N}.$$
So $v_{\delta'}$ is adjacent to $v_{\delta''}$. Now $v_{\delta''}$ satisfies conditions of Lemma~\ref{lem:14}. Therefore there is a path of length three between $v_{\delta}$ and $v_{\tau}$.

\item[3)] $k=4$ and $p=2$. Let $\delta'=[1^{n-8},4^2]$.
$$\frac{|C_{A_n}(\delta)|}{|C_{A_n}(\delta')|}=\frac{2^4.4!(n-8)!}{4^2.2!(n-8)!}\in \mathbb{N}.$$
This means $v_{\delta}$ is adjacent to $v_{\delta'}$. Also let $\delta''=[1^{n-4},2^2]$.
$$\frac{|C_{A_n}(\delta'')|}{|C_{A_n}(\delta')|}=\frac{2^2.2!(n-4)!}{4^2.2!(n-8)!}\in \mathbb{N}.$$
So $v_{\delta'}$ is adjacent to $v_{\delta''}$. By Lemma~\ref{lem:15} there is a path of length two between $v_{\delta''}$ and $v_{\tau}$

\item[4)] $1<k<4$ and $p=2$. In this case we have $n-2\leq kp\leq 6$ which is a contradiction with the assumption that $n\geq 9$.

\item[5)] $1<k<4$ and $p\neq 2$. Since $p$ is odd, $\delta'=[1^{n-kp+p},p^{(k-1)}]$ is an even permutation.
In this case there exists a positive integer $s$ such that
$$\frac{|C_{A_n}(\delta')|}{|C_{A_n}(\delta)|}=\frac{\frac{1}{2}(k-1)!p^{k-1}(n-kp+p)!}{\frac{1}{2}k!p^k(n-kp)!}=\frac{(n-kp+p)!}{kp(n-kp)!}={n-kp+p\choose p}s\in \mathbb{N}.$$
Therefore $|\delta'^{A_n}|$ divides $|\delta^{A_n}|$. So $v_{\delta}$ is adjacent to $v_{\delta'}$. Again
according to Lemma~\ref{lem:15} there is a path of length two between $v_{\delta'}$ and $v_{\tau}$.

\item[(iii)] $\delta=[1^{n-kp},p^k]$, $p\neq3$ and $k=1$.

By Lemma~\ref{lem:11}, together with our earlier assumption $kp>n-3$, the vertex $v_{\delta}$ is isolated.
\end{proof}
\begin{corollary}\label{cor:14}
$D(A_n)$ has at most three connected components. If it is disconnected, then two of its connected
components are $K_1$.
\end{corollary}
\begin{proof}
We know that for any positive integer $n$, at most two of the positive integers $n$, $n-1$ and $n-2$ are primes.
Hence by Theorem~\ref{thm:13} and Remark~\ref{rem:12}, we obtain the result.
\end{proof}
\begin{remark}
By using the fact that the distance between
any vertices of $D(S_n)$ and $v_{\tau}$ is at most $4$ (see proof of Theorem~\ref{thm:9} and Remark~\ref{rem:0}) we can find
that $diam(D(S_n))\leq 8$. A similar argument and using the proof of
Theorem~\ref{thm:13}, shows that $diam(D(A_n))\leq 10.$
\end{remark}
By considering $D(S_n)$ and $D(A_n)$ for some values of $n$ we may pose the following conjecture.
\begin{conj} The best upper bound for the diameter of $D(S_n)$ and $D(A_n)$ is $4$.
\end{conj}
\section*{Acknowledgment}
The authors would like to thank the anonymous referee for
helpful comments which improved the quality of this paper.

\bibliographystyle{amsplain}

\end{document}